\documentclass[reqno,10pt]{amsart}

\usepackage[english]{babel}
\usepackage[pdftex]{graphicx}
\usepackage{graphicx}
\usepackage{color}
\usepackage{hyperref}
\usepackage[dvipsnames]{xcolor}
\usepackage{mathtools}
\usepackage{comment}
\usepackage{amsmath}
\usepackage{amssymb}
\usepackage{amsfonts}
\usepackage{amsthm}
\usepackage{mathrsfs}

\newcommand{\N}{\mathbb{N}}
\newcommand{\R}{\mathbb{R}}

\newcommand{\Z}{\mathbb{Z}}

\newcommand{\f}{\rightarrow}

\newcommand{\JSJ}{\textit{JSJ}}

\newcommand{\fix}{\mathsf{Fix}}
\newcommand{\ax}{\mathsf{Axis}}

\newcommand{\ent}{\operatorname{Ent}}

\newcommand{\Ent}{\operatorname{Ent}}
\newcommand{\sys}{\operatorname{sys}}
\newcommand{\diam}{\operatorname{diam}}

\newcommand{\inj}{\operatorname{inj}}

\newcommand{\id}{\operatorname{id}}
\newcommand{\Vol}{\operatorname{Vol}}

\newcommand{\Ric}{\operatorname{Ricci}}

\newtheorem{thm}{Theorem}[section]

\newtheorem*{PdT}{Prime decomposition Theorem}
\newtheorem*{kneser}{Kneser's Conjecture}
\newtheorem*{JSJdT}{JSJ-decomposition Theorem}
\newtheorem*{dico}{Dicothomy}

\newtheorem{prop}[thm]{Proposition}
\newtheorem{lem}[thm]{Lemma}
\newtheorem*{lem*}{Lemma}
\newtheorem{cor}[thm]{Corollary}

\newtheorem*{fatto}{Fact}
\theoremstyle{remark}
\newtheorem{exmp}[thm]{Example}
\newtheorem{rmk}[thm]{Remark}

\theoremstyle{definition}
\newtheorem{defn}[thm]{Definition}

\begin{document}


\title[Local topological rigidity  of non-geometric $3$-manifolds]{Local topological rigidity \\ of non-geometric $3$-manifolds} 
\author[F.Cerocchi]{Filippo Cerocchi}
\thanks{This work has been completed while the first author was a postdoctoral fellow at the Mathematics Department in Rome, Sapienza. He is in debt as well to the CRM \lq\lq E. De Giorgi", SNS, Pisa and to the MPIM, Bonn, where he is currently a postdoctoral fellow.\\}
\address{F. Cerocchi, Max-Planck-Institut f\"ur Mathematik, Vivatsgasse 7, 5311, Bonn.}
\author[A. Sambusetti]{Andrea Sambusetti}
\address{A. Sambusetti, Dipartimento di Matematica \lq\lq G. Castelnuovo'', Sapienza Universit\`a di Roma, P.le Aldo Moro 5, 00185, Roma.\newline}
\email{\newline fcerocchi@mpim-bonn.mpg.de\,;\newline fcerocchi@gmail.com\,;\newline sambuset@mat.uniroma1.it}
\date{May, 2017}

\maketitle

\small

\vspace{-4mm}
{\bf Abstract.} We study Riemannian metrics on compact, torsionless, non-geometric \linebreak $3$-manifolds,   i.e. whose interior does not support any of the eight  model geometries. \linebreak We prove a lower bound ``\`a la Margulis'' for the systole  and a volume estimate for these manifolds, only in terms of an upper bound of entropy and diameter. We then deduce corresponding  local   topological rigidy results  
in the class $\mathscr{M}_{ngt}^\partial (E,D) $ of compact non-geometric 3-manifolds with torsionless fundamental group (with possibly  empty, non-spherical boundary) whose entropy and diameter are bounded respectively by $E, D$. For instance,  this class locally contains only finitely many  topological  types; and closed, irreducible manifolds in this class which are close enough (with respect to $E,D$) are diffeomorphic. Several examples and counter-examples are produced to stress the differences with the geometric case.
\normalsize

\tableofcontents

${}$
\vspace{-20mm}


\section{Introduction}

  Compact, differentiable $3$-manifolds (with or without boundary)  naturally fall into two main mutually exclusive classes: {\em geometric} manifolds, a chosen few,  whose interior supports a complete metric locally isometric to one of  the eight complete, maximal, homogeneous $3$-dimensional geometries
	\footnote{We use here  the term ``geometric'' as in the original definition given in \cite{thur82}; in the  case of manifolds with boundary, variations on this definition are possible and suitable for other purposes (i.e. uniqueness of the model geometries on each piece),  see for instance \cite{Bon}.},  and  {\em non-geometric} manifolds.
	These latter, by the solution of the Geometrization Conjecture,  are either  punctured $3$-spheres, or non-prime manifolds, or  irreducible with  non-trivial JSJ splitting; this has  interesting consequences on the structure of their fundamental group, as we shall see later (notice that also closed Sol-manifolds have a non-trivial JSJ decomposition, but this splitting does not have exactly the same properties as  in the non-geometric case, see  discussion in Section \S4). \pagebreak
	
	In the last  thirty years much effort has been made  to understand  the  model geometries  supported by the pieces of the JSJ-decomposition of  irreducible \linebreak 
	 3-manifolds  (notably, of  atoroidal 3-manifolds)  and {\em special} metrics  on general  \linebreak
	 $3$-manifolds (mostly  because of the simplification of the curvature tensor in dimension 3);  for instance, and by no means claiming to be exhaustive, the works on asymptotically harmonic metrics  \cite{schroeder-shah}-\cite{heber-knieper-shah},  works on nonnegatively Ricci curved metrics \cite{schoen-yau}-\cite{anderson-rodriguez}-\cite{shi}) and, last but foremost, on the Ricci flow  \cite{hamilton}-\cite{librobesson}. This has  led to amazing results,  such as Hamilton's elliptization of manifolds with positive Ricci curvature,   and  culminated in Perelman's solution of the Geometrization and Poincar\'e conjectures.
	
	The Riemannian geometry of non-geometric manifolds, or families of Riemannian metrics on them, deserved considerably less  attention,  in spite of their topological  peculiarities and their genericity:  non-geometric manifolds are very easy to produce, starting form hyperbolic or Seifert-fibered pieces, and  this class encompasses, for instance, the class of all {\em graph manifolds}\footnote{A \textit{graph manifold} is an irreducible $3$-manifold having a non-trivial JSJ-decomposition whose JSJ-components are all Seifert fibered (see \S4.1).}.
	This can be  explained by the  lack of any possible ``best metric'' on this class. Some remarkable exceptions are Leeb's work \cite{leeb} on the existence  of nonpositively curved metrics on aspherical $3$-manifolds, with or without boundary; or  Kapovich-Leeb's   \cite{KL}  and  Behrstock-Neumann \cite{BN2} results on  quasi-isometric rigidity  and quasi-isometry classification of non-geometric manifolds,  and other works on the restricted class of  Seifert and  graph manifolds (for instance \cite{Scott}, \cite{Bon} and \cite{ohshika},  and \cite{BN}, \cite{neu2},   \cite{frigerioLS} for graph manifolds and their higher-dimensional counterparts),   which are however mostly topological in  spirit.
	
	This paper, and  the forthcoming \cite{CSfiniteness}, are entirely devoted to the Riemannian geometry of {\em non-geometric} $3$-manifolds. We want to point out from the outset that all of our results on non-geometric 3-manifolds do not extend to geometric manifolds, as we shall show in each case, with possibly the exception of the class of $3$-manifolds of hyperbolic type, where the possibility of an extension is an interesting open question.
Our first result  is an estimate  \`a la Margulis for compact, non-geometric 3-manifolds with torsionless fundamental group. The original Margulis' Lemma  (established for non-positively curved manifolds $X$ with bounded sectional curvature, and  then generalized by the works of Fukaya-Tamaguchi \cite{fu-ya} and Cheeger-Colding \cite{ch-co} and by Kapovich-Wilking \cite{ka-wi} to manifolds with only a lower Ricci curvature bound), concerns the virtual nilpotency of the subgroup of $\pi_1(X)$ generated by sufficiently small loops at any point $x \in X$. 
For compact, negatively curved manifolds,  this  yields  an estimate of the systole, or of the injectivity radius, in terms of bounds of the sectional curvature and of the diameter (see, for instance, \cite{bur-za}):
\vspace{-2mm}

$$\sys \pi_1 (X) = 2 \inf_{x \in X} \inj(x) \geq \frac{\epsilon_0(n)}{K \cdot  \sinh^{n-1} D}$$

\noindent for any $n$-manifold $X$ with  $-K^2 \leq K_X < 0$ and diameter bounded by $D$.\\
 A similar result,  based more on topological  arguments than on the analysis of the curvature tensor,
	is Zhu's estimate of the contractibility radius for 3-manifolds under controlled   Ricci curvature, diameter and volume (\cite{Zhu}).\\
The systolic estimate we give, for non-geometric 3-manifolds, ignores  curvature, and only uses a normalization by the entropy:

\begin{thm}
\label{app_MTHM_1}${}$ 
	Let $X$ be any  compact, non-geometric Riemannian  $3$-manifold,  with torsionless fundamental group and no spherical boundary components. 
	Assume  that   $\ent(X)\le E$ and that $\diam(X)\le D$: then,  
	\begin{equation}\label{sys_3mfd}
	\sys\pi_1(X)\ge s_0 (E,D) :=\frac{1}{E}\cdot\log\left(1+\frac{4}{e^{26\,E\,D}-1}\right)
	\end{equation}
\end{thm}

\noindent Recall that the (volume-){\em entropy} of a compact Riemannian  manifold $X$ is the exponential growth rate of the volume of balls in the universal covering $\tilde{X}$:
\vspace{-3mm}

\begin{equation}\label{entropy}
\Ent (X) = \limsup_{R \rightarrow \infty} R^{-1} \cdot  \log \Vol B_{\tilde X} (\tilde x,R)
\end{equation}

\vspace{-1mm}
\noindent for any choice of $\tilde x \in \tilde X$. Actually, the lift $\tilde{\mu}$ of any finite Borel measure $\mu$ on $X$  can be used in the above formula, obtaining  the same result, cp.  \cite{sam2}. In particular, using the measure $\mu= \sum_{g \in G} \delta_{g \tilde x}$  given by the sum of Dirac masses of one orbit of $G \cong \pi_1(X,x)$ on $\widetilde X$, one sees that the entropy gives the exponential growth rate of pointed homotopy classes of loops in \nolinebreak $X$ (where the length of  classes is measured by the shortest  loop  in the class).  Moreover, it is well known that this  also equals, in non-positive curvature,  the {\em topological entropy} of the geodesic flow on the unitary tangent bundle of $X$, cp. \cite{manning}).   
For closed manifolds,  a lower  bound  of the Ricci curvature $\Ric_X \geq -(n-1)K^2$ implies  a corresponding upper bound of the entropy $\Ent(X) \leq (n-1)K$,  by the classical volume-comparison theorems of Riemannian geometry.
However, entropy  is a much weaker  invariant than Ricci curvature; 
actually, $\Ent(X)$ can be seen as an averaged version of the curvature (this can be given a precise formulation in negative curvature by  integrating  the Ricci curvature on the unitary tangent bundle of $X$ with respect to a suitable measure, cp. \cite{knieper}), and only depends on the large-scale geometry of $X$. 
\vspace{2mm}

Theorem \ref{app_MTHM_1} stems from the interplay between the metric structure and the algebraic properties of $\pi_1(X)$, given by the Prime Decomposition Theorem and the JSJ-decomposition Theorem for irreducible $3$-manifolds. We shall see in Section \S\ref{sectionsystolic} a  more general estimate for manifolds whose fundamental group acts acylindrically on a simplicial tree (which generalizes some estimates in \cite{Cer}).  


\begin{rmk}\label{ngisnecessary}
	The assumption \lq\lq non-geometric\rq\rq\, in Theorem \ref{app_MTHM_1} is necessary.  \\
	Besides the four geometries of sub-exponential growth $\mathbb S^3$,  $\mathbb S^2\times\R$, $\mathbb E^3$ and $Nil$, where it is evident that a simple bound on the diameter does not force any lower bound of the systole, we shall see in section \S5 that  every closed $3$-manifold modelled on $Sol$, $\mathbb H^2\times\R$ or  $\mathbb H^2\widetilde\times \R$   also  admits a sequence of metrics $g_{\varepsilon}$  such that $\ent(X,g_\varepsilon)\le E$, $\diam(X,g_\varepsilon)\le D$ and $\sys\pi_1(X,g_\varepsilon)\f 0$. In all the examples, with the exception of  $\mathbb H^2\widetilde\times \R$, the metrics $g_{\varepsilon}$ are even locally isometric to the respective model geometries. \\
	In contrast, such a family of  metrics cannot  be found on a fixed, closed $3$-manifold $X$ of hyperbolic type; actually, a hyperbolic metric $g_0$ being fixed on $X$ (recall that by Mostow's rigidity Theorem this metric is unique   up to isometries), then the  systole of {\em any} other Riemannian metric $g$ on $X$ is bounded away from zero in terms of its entropy and diameter,  and of the injectivity radius of $(X, g_0)$, in view of the results in \cite{BCG}.
	It is not known to the authors if it is possible to find a universal lower bound as in  (\ref{sys_3mfd}), holding for  Riemannian metrics on {\em all}  closed $3$-manifolds of hyperbolic type. \pagebreak
	 
\end{rmk}

\begin{rmk} 
	Also, the torsionless assumption in Theorem \ref{app_MTHM_1} cannot be dropped. \linebreak
	For any closed $3$-manifold  $X$  and any  $p\geq 2$, one can construct on  the connected sum $Y=X \#(\mathbb{S}^3/\Z_p)$ with a lens space  a family of metrics $g_{\varepsilon}$,  
	with $\epsilon \rightarrow 0$,  such that $\diam(Y, g_{\varepsilon})\le D$, $\ent(Y, g_{\varepsilon})\le E$ and $\sys(Y, g_{\varepsilon})=\varepsilon$ (see   \cite{Cer}, Example 5.4).
\end{rmk}

The assumption on the boundary in Theorem \ref{app_MTHM_1} can be relaxed by asking that $X$ does not have  the homotopy type of a punctured, geometric manifold; notice that one can excise an arbitrarily small ball from a  geometric manifold without  modifying the fundamental group and the systole, and this gives an easy counterexample to (\ref{sys_3mfd}) for punctured geometric manifolds.

\noindent As an immediate consequence of  (\ref{sys_3mfd}) and of Gromov's  systolic inequality for essential manifolds (\cite{gro_frm}, Theorem 0.1.A) we deduce  the following  volume estimate:

\begin{cor}
\label{app_MTHM_1v}
	Let $X$ be any closed, non-geometric Riemannian $3$-manifold  \linebreak with  torsionless fundamental group, which is not homeomorphic to the connected sum of a finite number of copies of $S^2\times S^1$.
	Assume  that $\ent(X)\le E$ and that $\diam(X)\le D$: then,
	\begin{equation}\label{vol_3mfd}
	\Vol(X)\ge C \cdot s_0 (E,D)^3
	\end{equation}
\end{cor}

It is worth to stress   that the volume estimate holds in particular for any non-geometric closed  \textit{graph manifold} (\textit{i.e.} any graph manifold which is not a \textit{Sol}-manifold) and for connected sums of such manifolds, with the remarkable exception of  connected sums of copies of $S^2\times S^1$. The volume estimate above 
is particularly interesting in these cases because, for graph manifolds (and connected sums of graph manifolds), the simplicial volume vanishes (see \cite{Soma}, Corollary 1) and   it is thus impossible   to obtain estimates for the volume via the classical arguments of bounded cohomology.

\begin{rmk} \label{exception}
	The exception of a connected sum  of copies of $S^2\times S^1$  in Corollary \ref{app_MTHM_1v}  \linebreak  cannot be avoided. In section \S5, Ex. \ref{collapsingS2S1}, we shall exhibit  a family of metrics $g_{\varepsilon}$ on  $X =\#_k(S^2\times S^1)$,  for any $k \geq1$, with  $\lim_{\varepsilon\f0}\Vol(X , g_\varepsilon)=0$  while,  for all $\epsilon>0$,
	$$\ent(X , g_{\varepsilon} )\le E ,\;\;\diam(X , g_{\varepsilon} )\le D ,\;\;\sys\pi_1(X,  g_\varepsilon )\ge s $$ 
\end{rmk}

\vspace{1mm}
The systolic estimate \ref{app_MTHM_1}  is the keystone of  the local topological rigidity and finiteness results that we shall prove in  Section \S\ref{sectiondim3}.
Namely, consider the classes $$\mathscr M_{ngt} (E,D)  \;\;\;\;\;\; \mbox{(respectively, $\mathscr M^{\partial}_{ngt} (E,D)$ )} $$  of  closed 
(resp. compact,  with  possibly empty boundary,  and no spherical  boundary components)  
  connected, {\em non-geometric} Riemannian $3$-manifolds $X$, 
 {\em with torsionless fundamental  group},   
 whose entropy and diameter are respectively bounded by $E$ and $D$, 
 endowed  with  the Gromov-Hausdorff distance $d_{GH}$. 
	 Recall that, in restriction to  oriented,  irreducible $3$-manifolds  $X$, the following are equivalent:
	
	(i) $X$ is a  $K(\pi, 1)$-space;
	
	(ii) $X$ has torsionless fundamental group;
	
	(iii) $X$ has infinite fundamental group; 
	
	(iv) $X$ is not a quotient of $S^3$.\\
	(The  implication  (i) $\Rightarrow$ (ii) is  standard, see for example \cite{hatcher} Prop.2.45, while \linebreak (ii) $\Rightarrow$  (iii) $\Rightarrow$ (iv)   are trivial;
	  on the other hand, (iv) $\Rightarrow$ (iii)  follows from Perelman's Elliptization Theorem, and (iii)   $\Rightarrow$ (i) is consequence of the JSJ-decomposi\-tion and of the classification of Seifert fibered manifolds.)
	
    The topological type of geometric manifolds, possibly with the exception of manifolds of hyperbolic type, enjoys a lot of freedom under Gromov-Hausdorff convergence: one can easily produce geometric manifolds which are arbitrarily close in the Gromov-Hausdorff distance, while being very different. 
	For instance, 
	 the quotient  of the Heisenberg group or of the $Sol$-group by  the  respective integral lattices $H^3_{\mathbb{Z}}$ and   $Sol_{\mathbb{Z}}$   admit metrics which make them  arbitrarily close 
to a flat $3$-torus,  since all of them can collapse with bounded curvature (and, a fortiori, with bounded entropy) to a circle; similar examples can be produced by taking  a surface of hyperbolic type $\Sigma_g$, and considering  its unit tangent bundle $U\Sigma_g$ and  the product $\Sigma_g \times \mathbb{S}^1$, which both can collapse with bounded curvature to $\Sigma_g$ (see Ex. \ref{excoll}). \linebreak
 Non-geometric manifolds (though often also collapsible, since graph manifolds admit the so called {\em $F$-structures}, \cite{chegro})  are more topologically rigid, as their  topological type is locally determined,  provided that  the entropy stays bounded while  approaching some fixed manifold: 
	
	\begin{thm}\label{rigidity} There exists $\delta_0 \!\!= \delta_0 (E,D) \!>\!0$ such that  for any $X,\!X' \! \in\!  \mathscr M^{\partial}_{ngt} (E,D)$:\\
		(i) if $d_{GH}(X,X') < \delta_0$, then $\pi_1(X)\cong\pi_1(X')$;\\
		(ii) if $X, X'$ are irreducible 
		and $d_{GH}(X,X') < \delta_0 $,  then  $X$ and $X'$ are homotopically equivalent. 
		(One can take $\delta_0 = \frac{1}{40} s_0 (E, D)$, for   $s_0 (E, D)$ as in Theorem \ref{app_MTHM_1}).
	\end{thm}

	%
	%
	
     This theorem might be reminiscent of Kapovich-Leeb quasi-isometric (virtual) rigidity results for  the fundamental group of  non-geometric $3$-manifolds \cite{KL0}. However, besides the stronger conclusions 
	(the fundamental group  cannot be determined simply from the quasi-isometry type),
	notice that, without any assumption on the entropy,  one can easily produce non-geometric manifolds  $X$, $X'$ which are arbitrarily close in the Gromov-Hausdorff distance and which  do not have quasi-isometric fundamental groups.  
	Take,  for instance,  the connected sum of an irreducible manifold $X$ with  any,   arbitrarily small in size,  non-simply connected 3-manifold $M$; then, the fundamental group of the resulting manifold $X'=X \# M$  
	is  not quasi-isometric to $\pi_1(X)$, by \cite{PW}.
	Also,  it is well-known that any two closed graphs manifolds have quasi-isometric fundamental group (cp. \cite{BN}), while being far from having isomorphic fundamental   groups.
\vspace{1mm}	
	 
	The fundamental group completely determines the integral homology groups of  closed (connected) orientable $3$-manifolds, as  $H_0(X,\Z) =H_3(X,\Z) = \Z$,  $H_1(X,\Z) = \pi_1(X)/[\pi_1(X,),\pi_1(X)]$ and 
	$H_2(X,\Z) = H^1(X,\Z = H_1(X,\Z)/tor$; thus, in restriction to the subset $\mathscr M_{ngt}(E,D) $, the local rigidity of the fundamental group implies the local constancy  of all homology groups.
	However, by Swarup's finiteness theorem for irreducible $3$-manifolds with given fundamental group and by Kneser's Conjecture, Theorem \ref{rigidity} (i)  has the following stronger consequence: 
		

	\begin{cor}\label{topfiniteness} The diffeomorphism type is locally finite  on the space  $\mathscr M^\partial_{ngt} (E,D)$. 
	\end{cor}
	
	
	Recall that, if $X$ and $X'$ are two closed $3$-manifolds with torsionless fundamental group, then they are homotopy equivalent if and only if they are homeomorphic\footnote{This is no longer true if we assume the manifolds to have non-trivial boundary (even for irreducible manifolds with incompressible boundary) see \cite{jo2} and \cite{swa}.}, if and only if they are diffeomorphic.
	The first equivalence is a consequence of the solution of the Borel Conjecture for closed $3$-manifolds with torsionless fundamental group, which follows from the work of Waldhausen (\cite{Wal}) for Haken $3$-manifolds, and from the work of Turaev \cite{tur} and Perelman's solution  of  the Geometrization Conjecture,  for non-Haken $3$-manifolds. 
	The second equivalence  follows from the work of Moise, Munkres and Whitehead (\cite{Moi}, \cite{Mun1}, \cite{Mun2}, \cite{Whd}) and holds for any $3$-manifold,  even  without the torsionless and closeness assumption. 

	From   Theorem \ref{rigidity} (ii)    we also deduce the following, more explicit: 
	
	\begin{cor}\label{diffrigidity}  
		For  all $X,X' \in \! \mathscr  M_{ngt} (E,D)$ with $X$ irreducible,  if $d_{GH}(X,X') \leq \delta_0$ \linebreak then $X'$ is  diffeomorphic to $X$ (for $\delta_0=\delta_0 (E, D)$ as in Theorem \ref{rigidity}).
	\end{cor}


Notice that Corollary \ref{diffrigidity}  shows, in particular,  that the Gromov-Hausdorff distance defines a metric (quotient) structure on the diffeomorphisms classes of irreducible manifolds in   $\mathscr  M_{ngt} (E,D)$; this  is false   for reducible manifolds:

\begin{rmk}\label{primeisnecessary}
	Irreducibility  in Theorem \ref{rigidity} (ii) and Corollary \ref{diffrigidity} is necessary. \\
	 We shall see in the Ex. \ref{nonnprig} a pair of  closed, non-geometric, non-homotopically equivalent  
	 $3$-manifolds $Y$ and  $\bar Y$, which   admit  sequences of metrics $(g_n)_{n \in \mathbb{N}}$, $(\bar g_n)_{n \in \mathbb{N}}$ with uniformly bounded entropy and diameter, such that the  Gromov-Hausdorff distance between $(Y, g_n)$ and $(\bar Y, \bar g_n)$ goes to zero when $n \rightarrow \infty$.
\end{rmk}





\vspace{1mm}
These results should be compared to general finiteness and convergence theorems in Riemannian geometry, under  classical  curvature, diameter, and volume (or injectivity radius) bounds.
In particular,  Corollary \ref{diffrigidity} can be interpreted as  a quantitative version (in restriction to  non-geometric  3-manifolds with infinite fundamental group)  of Cheeger-Colding celebrated diffeomorphism Theorem \cite{ch-co}, saying that if a sequence of smooth $n$-manifolds $X_k$, with  Ricci curvature uniformly bounded from below, tends in the Gromov-Hausdorff convergence  to a smooth \linebreak  $n$-manifold $X$,  then $X_k$ is diffeomorphic to $X$ for $k \gg 0$. 
Notice however that, despite the  restricted class of application, our results  only need a control of a much weaker invariant than Ricci curvature:  it is easy to exhibit convergent families of Riemannian manifolds with  bounded entropy, where the Ricci curvature is not uniformly bounded  (see \cite{rev} for some enlightening examples). Also, Cheeger-Colding's diffeomorphism theorem does not apply without the strong assumption that the limit space is a manifold, 
whereas  Corollary \ref{diffrigidity} shows that the $X_k$'s are always diffeomorphic for $k \gg0$.
 In this perspective,  it is somewhat surprising that, for non-geometric manifolds,
a bound on  the entropy suffices to capture the local topological type, 
and actually does a better service than a  Ricci curvature bound  in the case of manifolds with boundary (notice in fact that we do not need any supplementary curvature  assumption on the boundary).
\vspace{1mm}

	Finally, let us  state the following finiteness theorem under Ricci curvature bounds, as an immediate corollary  of Theorem \ref{diffrigidity}   and Gromov's precompactness theorem (or, equivalently, of the volume estimate
 (\ref{app_MTHM_1v})  and Zhu's homotopy finiteness theorem, cp.  \cite{Zhu}, Theorem 1):
	
	\begin{cor}\label{fnt_glo}
		Let $\mathscr M_{ngt} (Ric_K, D)$ be the family of closed, non-geometric,\linebreak Riemannian $3$-manifolds with torsionless fundamental group, satisfying  the bounds
$\Ric\ge -(n-1)\,K^2$ and $\diam \le D$.
		The number of diffeomorphism types in $\mathscr M_{ngt} (Ric_K,D)$ is finite.
	\end{cor}

	Comparing with  Zhu's  theorem,  we are dropping the lower bound assumption on the volume;  we pay this choice by restricting ourselves to the set of torsionless non-geometric $3$-manifolds.
	We believe that a similar  finiteness result  should hold  for non-geometric manifolds  satisfying only  a bound on entropy instead of Ricci curvature; 
	this point of view will be developed elsewhere by the authors  \cite{CSfiniteness}.
\vspace{2mm}

\footnotesize
{\sc Aknowledgments.}  We are in debt  with G. Courtois and R. Coulon for several valuable discussions, and with S. Gallot for his attention to this work and encouragement.
\normalsize


\section{Nonabelian, rank 2  free subgroups}\label{sectionfree}

In this section we recall some facts about $k$-acylindrical actions of groups on simplicial trees. The aim is to give quantitative results on the existence of  $2$-generators free subgroups starting from two  prescribed elliptic or hyperbolic generators.
 \vspace{1mm}

\noindent We recall that, given a group $G$ acting by automorphisms on a  tree $\mathcal T$ {\em without  edge inversions} ( i.e. no element  swaps the vertices of some edge),  the elements of $G$ can be divided into two classes: elliptic and hyperbolic elements. They can be distinguished by their \textit{translation length}, which is defined,  for  $g\in G$,  as 
$$\tau(g)=\inf_{\mathsf v\in\mathcal T}d_{\mathcal T}(\mathsf v,g\cdot\mathsf v)$$
  where $d_{\mathcal T}$ denotes the simplicial distance of ${\mathcal T}$, i.e. with all edges of unit length.
 \linebreak
If $\tau(g)=0$ the element $g$ is called \textit{elliptic}, otherwise it is called \textit{hyperbolic}. \\
We shall denote by $\fix (g)$ the set of fixed points of an elliptic element $g$,  and by $T(g) = \bigcup_{n \in \Z^\ast} \fix(g^n)$ the set of  points which are fixed by some non-trivial power of $g$; these are  (possibly empty) connected subtrees of $\mathcal T$. If $h$ is a hyperbolic element then $\fix (h)=\varnothing$ and $h$  has a unique axis on which it acts by translation, denoted $\ax(h)$; each element on the axis of $h$ is translated at distance $\tau(h)$ along  the axis, whereas elements at distance $\ell$ from the axis are translated of $\tau(h)+2\ell$. 

\noindent  Let $\mathcal T_G$  be the minimal subtree of  $\mathcal T$ which is $G$-invariant: the action of  $G$   is said {\em elliptic} it  $\mathcal T_G$  is a point, and {\em linear} if $\mathcal T_G$ a line; in both cases we shall say that the action of $G$ is {\em elementary}.
 We also recall that  an action without edge inversions is called \textit{$k$-acylindrical} if the set $\fix(g)$  has diameter less than or equal to $k$,  for any  elliptic $g\in G$. 
 The notion of $k$-acylindrical action on a tree is due to Sela (\cite{Sela}), and arises naturally in the context of Bass-Serre theory, as we shall see later.
 \vspace{3mm}
 
 Groups acting $k$-acylindrically on trees are well-known to possess free subgroups. We need a quantitative version of this, estimating, for every prescribed, non-commuting pair of elements $g_1,g_2$, the maximal length of a word in $g_1,g_2$  generating with $g_1$ (or with some bounded power of $g_1$) a free sub(semi-)group:

\begin{thm}[Quantitative free product subgroup theorem]\label{freesmgr} ${}$\\
Let $G$ be a 
group  acting  $k$-acylindrically  
on a simplicial tree $\mathcal T$:  
\vspace{1mm}

\noindent (i) if $g_1, g_2\in G$ are elliptic  and $\fix(g_1)\cap\fix(g_2)=\varnothing$, then the group   $\langle g_1, h^p \,g_1\,h^{-p} \rangle$ is  a rank $2$  free product, for $h= g_1g_2$ and $p \geq (k+1)/2$;
\vspace{1mm}

\noindent (ii) if $g\in G$ is elliptic and  $h\in G$ is hyperbolic, then 
 the group   $\langle  g, h^{p}g\ h^{-p} \rangle $ is a   rank $2$ free product, for $p\ge k+1$; 

\vspace{1mm}

\noindent (iii) if $h_1, h_2\in G$ are  hyperbolic with $\ax(h_1) \neq  \ax (h_2)$, then: \\
-- if $\diam \left( \ax(h_1) \cap \ax (h_2) \right)   \leq 3k$, then $ \langle  h_1^q, h_2^q \rangle $   is rank $2$ free subgroup,  \linebreak for $q \ge 3k+1$; \\
--  if $\diam \left( \ax(h_1) \cap \ax (h_2) \right)   > 3k$, then either $\langle h_1, h_2^{p}h_1h_2^{-p} \rangle $ or  $\langle  h_2,h_1^{p}h_2h_1^{-p}  \rangle $ is  a  rank $2$ free subgroup,   for  $p\ge 3$;\\
-- in any case (even without the assumption of $k$-acylindricity)  either $\{h_1, h_2\}$ or  $\{h_1^{-1}, h_2\}$  freely generate  a  rank $2$  free semigroup. 
\end{thm}


In order to prove Theorem \ref{freesmgr},  we shall need the  following basic facts  (cp.  \cite{butal}, \cite{kaweid}):
\pagebreak

\begin{lem}\label{fact} Let $g_1, g_2 $ be elliptic elements  of a group $G$ acting without edge inversions 
on a simplicial tree $\mathcal T$: \\
(i) if   $\fix(g_1) \cap \fix(g_2) = \emptyset$, then $g_1g_2$ is hyperbolic with translation length 
$$\tau(g_1g_2)= 2 d_{{\mathcal T}} (\fix(g_1), \fix(g_2)) \;;$$
(ii) if   $T(g_1) \cap T(g_2) = \emptyset$, then the group $\langle g_1,  g_2\rangle$  is a rank $2$  free product. 

\end{lem}

\begin{lem}\label{fact2} 
Let $g_1, g_2 $ be hyperbolic elements  of a group $G$  acting without edge inversions  on a simplicial tree $\mathcal T$, and let $J=\ax(h_1) \cap \ax (h_2)$: if 
$$\diam ( J ) <  n\,\min\{\tau(h_1),\,\tau(h_2)\}$$
then   $h_1^n$ and $h_2^n$ generate a nonabelian, rank $2$  free subgroup of $G$.
\end{lem}

\begin{proof}[Proof of Theorem \ref{freesmgr}]  To prove (ii) it is sufficient, by Lemma \ref{fact} (ii),  to show that  $T(g) \cap T(g') = \emptyset$,  for $g'=h^{p}g\ h^{-p}$, and $p\ge k+1$.  This is equivalent to show that  
$ \fix(g^{\ell_1})\cap\fix( g'^{\ell_2} )=\varnothing$  for all $\ell_1,\ell_2\in\Z^\ast$.
As  $\fix(g^\ell)\supseteq\fix(g)$ for any  $\ell\in\Z^*$, this last condition  is equivalent to: 
\begin{equation} \label{fixl}
\fix(g^{\ell})\cap\fix(h^{p}g^{\ell}  h^{-p})=\varnothing,\;\;\;\; \forall \ell\in\Z^*
\end{equation} 
We  consider  the two cases:   $\fix(g^{\ell})\cap \ax(h)=\varnothing$ or $\fix(g^{\ell})\cap \ax(h)\neq\varnothing$.\\ 
\noindent In the first case  the projection of $\fix(g^{\ell})$ onto $\ax(h)$ is one point, denoted   $v_*$. Since $\fix(h^{p}g^{\ell}h^{-p})=h^{p}\cdot\fix(g^{\ell})$, then   $h^{p}\cdot v_*$ is the projection of $\fix(h^{p}g^{\ell_1\ell_2}h^{-p})$ onto $\ax(h)$. This implies that   (\ref{fixl}) holds  for all  $p>0$ as in this case
$$d_{\mathcal T}\left(\fix(h^{p}g^{\ell}h^{-p}),\fix(g^{\ell})\right) 
    \ge d_{\mathcal T}(v_*, h^{p} v_*) +2\
    \ge p\tau(h)+2 $$

\noindent In the second case, let $J=\fix(g^{\ell})\cap \ax(h)$ 
 and notice that $\diam(J) \leq k$ by $k$-acilindricity. 
So, let  $v_* \in J$  such that $d_{\mathcal T}(v_*,v)\le\frac{k}{2}$ for any $v\in J$; 
 observe that  $h^{p}\cdot v_*$ satisfies the same property with respect to the set $h^{p}(J)=\fix(h^{p}g^{\ell }h^{-p})\cap \ax(h)$.
\noindent Since $h$ acts by translation of $\tau(h)\ge 1$ on its axis, we have
\small
$$d_{\mathcal T}\left(\fix(h^{p}g^{\ell}h^{-p}),\fix(g^{\ell})\right) 
 \ge d_{\mathcal T}\left(v_*, h^{p}\cdot v_*\right)-\frac{k}{2}-\frac{k}{2}
\ge p\,\tau(h)-k
$$
\normalsize
  Since the action is $k$-acylindrical we conclude that, in this case,  condition (\ref{fixl}) is satisfied for all $\ell\in\Z^*$ if $p\ge k+1 $ (as $\tau(h)\ge 1$),  which proves part  (ii).

\noindent Assertion (i) follows by applying the above argument to $g = g_1$ and to $h = g_1g_2$, which is a hyperbolic element with   $\tau(h)\ge 2$, by Lemma \ref{fact}(i).

 \noindent To prove  (iii), we may assume that 
 $J=\ax(h_1)\cap \ax(h_2) \neq\varnothing$,
 otherwise  $h_1$ and $h_2$ have an evident ping-pong dynamics for every choice of base point $x_0 \in {\mathcal T}$, and they clearly generate a nonabelian, rank 2 free subgroup.\\
 If $d=\diam(J) \le 3k$,  then the elements $h_1^{q}$, $h_2^{q}$, for any $q\ge 3k+1$, generate a nonabelian, rank $2$ free subgroup by Lemma \ref{fact2}.  
 Assume now that $d\ge 3k+1$. \linebreak By the condition of $k$-acylindricity, we  infer that 
 $\max \{ \tau (h_1), \tau (h_2) \} > d/3 \;;$ otherwise, there  exists a connected subset $J' \subset J$, with 
 $\diam(J') = d/3  > k$, which is fixed by $h_1^{-1}h_2^{-1}h_1h_2$ 
 (actually, assume $J$ oriented by the translation direction of $h_1$: then,  it is enough to take $J'$ equal to the first subsegment of length $d/3$ of $J$,  if  $h_1, h_2$ translate $J$ in the same direction;   and $J'$ equal to the middle  subsegment of $J$ of length $d/3$,  when  $h_1,h_2$ translate in opposite directions).
So, we may assume  that $\tau(h_1)> d/3$: in this case, for $p\geq 3$ we have
$$ \ax(h_1^ph_2h_1^{-p}) \cap \ax (h_2) = (h_1^p .\ax(h_2)) \cap \ax (h_2) =\varnothing$$
hence $\{ h_2, h_1^ph_2h_1^{-p}\} $  generate a nonabelian, rank 2 free subgroup by Lemma \ref{fact}. 
The  case where $\tau(h_2)> d/3$ is analogous. The last assertion in (iii) is classical.
 \end{proof}


\section{Systolic estimates}\label{sectionsystolic}

\begin{defn}\label{EDMG}
Let $(G,d)$ be a discrete, proper metric group, \textit{i.e.} a discrete group $G$ endowed with a left-invariant distance such that the balls of finite radius are finite sets.
The \textit{entropy} of   $(G,d)$ is:
$$\Ent(G,d)= \limsup_{R\f\infty}\frac{1}{R}\log \, \# B_d (\id,R)$$
where $B_d(g, R) = \{g'\;|\; d(g,g')<R\}$ denotes the ball of radius $R$ centered at $g$.
\end{defn}

 \vspace{2mm}
We shall be mainly interested in two different kinds of distances on $G$:
\vspace{1mm}


\noindent \textit{-- word or word-weighted distances}, associated to  some  finite generating set $\Sigma$ and to some weigth function $\ell: \Sigma \rightarrow \R^+$,  denoted $d_\ell$;  this is the unique left-invariant length distance on the Cayley graph $\mathcal C(G,\Sigma)$ such that $d_{\ell}(\id,s)=\ell(s)$  and is linear on each edge (when $\ell=1$ this is the usual word metric $d_\Sigma$ associated with $\Sigma$);.
\vspace{1mm}

%
\noindent \textit{-- geometric distances,} associated to some discrete, free action of $G$ on a pointed, Riemannian manifold $(Y, y_0)$,  denoted $d_{y_0}$;  in this case  $d_{y_0}(g\,,g'\,)=  d(g.\,y_0,g'.\,y_0)$  is the distance between corresponding orbit points. 
\vspace{1mm}

\noindent We shall denote the corresponding distances from the identity by $|g|_\Sigma$, $|g|_\ell$, $|g|_{y_0}$.

%
%


\noindent The following  properties of the entropy are  well-known, and will be used later:

\begin{itemize}
\item[(E1)] When $Y=\widetilde{X}$ is the Riemannian universal covering of a Riemannian manifold $X$,  with $G \cong \pi_1(X)$ acting on $Y$ by deck transformations,  for any choice of $\tilde{x}_0 \in \tilde{X}$, the volume-entropy of $X$  satisfies  $\ent(X)\geq\Ent(G, d_{{\tilde x}_0})$,  with  equality  when $X$ is compact, cp. \cite{sam2}.
\item[(E2)] Given distances $d_1\leq d_2$ on $G$, we have: $\Ent(G,d_1)\geq\Ent(G,d_2).$ 
 
\end{itemize}

 \noindent The announced volume estimates of Theorem  \ref{app_MTHM_1} and Corollary \ref{app_MTHM_1v} are a particular case of the following result: 


\begin{thm}\label{MAINTH}
Let $X$ be any compact, connected Riemannian manifold with 
 \linebreak 
  torsionless 
 fundamental group, acting  non-elementarily and $k$-acylindrically  
  on a simplicial tree.  
If $\diam(X)\le D$\,, $\ent(X)\le E$, then: 
\begin{equation}\label{MAINest}
\sys\pi_1(X)\ge \frac{s_0 (E \cdot D)}{E}
\end{equation}
  where $\displaystyle s_0 (t)= \log\left(1+\frac{4}{e^{(4 k+ 10)\,t}-1}\right) $.
Moreover, if $X$ is $1$-essential  then: 
\begin{equation}\label{volume}
\Vol(X)\ge C_n\cdot \left( \frac{s_0 (E \cdot D)}{E} \right)^n
\end{equation}
\end{thm}
 
\noindent Recall that, following M. Gromov  \cite{gro_frm},  a $1$\textit{-essential} $n$-manifold $X$ is a closed, connected  $n$-manifold which admits a continuous map  into an aspherical space $f:X\f K$, such that the image of the fundamental class $[X] \in H_n(X, \Z)$ via the homomorphism induced in homology by $f$ does not vanish. 
\vspace{1mm}

In the proof of Theorem \ref{MAINTH}, we shall need the following, elementary:
\begin{lem}\label{gensys}
Let $G$ be any finitely generated group, acting 
without edge-inversions
on a simplicial tree $\mathcal T$, and let $\Sigma$ be any finite generating set for $G$:
\vspace{1mm}

\noindent (a)  if the action is non-elliptic, then there exists a hyperbolic element $h\in G$ such that $|h|_\Sigma\le 2$. Namely,  either  $h \in \Sigma$, or $h$ is the product of two elliptic elements $s_1, s_2\in\Sigma$  such that $\fix(s_1)\cap\fix(s_2)=\varnothing$; 
\vspace{1mm}

\noindent (b)  if the action is non-elementary, then for any hyperbolic element   $h\in G$ there exists  $s\in \Sigma$ which does not belong to the normalizer $N_G(h)$ of  $\langle h \rangle$ in $G$.
\vspace{1mm}

\noindent (c)   if the action is  linear  and acylindrical, then $G$ is virtually cyclic.

\end{lem}

\begin{proof}[Proof of Lemma \ref{gensys}] Let us show (a).
If $s\in \Sigma$ is a hyperbolic element, we choose $h=s$. On the other hand, if   $\Sigma$ only contains  elliptic elements,  there exists a pair of elements $s_1, s_2$, from $\Sigma$,  such that $\fix(s_1)\cap \fix(s_2)=\varnothing$, because $G$ acts on  $\mathcal T$ without global fixed points. 
Then,  $h=s_1s_2$ is a hyperbolic element with $|h|_\Sigma\le 2$.\\
 Let us now prove (b). Let  $h$ be  a hyperbolic element of $G$;  an element $s\in\Sigma$ belongs to $N_G(\langle h\rangle)$ if and only if it globally preserves $\ax(h)$. Therefore, if $s\in N_G(\langle h\rangle)$ for all $s\in\Sigma$, we would deduce that $G=N_G(\langle h\rangle)$ preserves a line, and thus the action is elementary, a contradiction.
For (c),  assume that $G$ preserves a line of  $\mathcal T$; this is the axis of some hyperbolic element $h$ with minimal displacement,  by (a). Any other element $s \in \Sigma$ either  is a hyperbolic element such that $\ax(s)=\ax(h)$, or is elliptic and  globally preserves $\ax(h)$, swapping the two ends. In the first case  $s$ is a power of $h$, by acylindricity. In the second case, $s$ acts on $\ax(h)$ as a reflection with respect to some vertex, and $s^2$ fixes pointwise  the  axis;  hence, again  by acylindricity,  $s^2=1$ and  $shs^{-1}=h^{-1}$. Also, if $s' \in \Sigma$ is another  elliptic element,   $ss'$ fixes the ends of  $\ax(h)$, hence it is again a power of $h$. It follows that  $G = \langle h\rangle \cong \mathbb{Z} $ or $G = \langle h,s\rangle   \cong \mathbb{Z} \rtimes \mathbb{Z}_2$.
\end{proof}

\begin{proof}[Proof of Theorem \ref{MAINTH}] 
 The volume estimate (\ref{volume}) follows from (\ref{MAINest}) just by applying   Gromov's  Systolic inequality  $\Vol(X)\ge C_n\cdot \left(\sys\pi_1(X)\right)^n$, which holds  for any  $1$-essential $n$-manifold, for  a universal constant $C_n$ only depending  on the dimension \nolinebreak $n$ (see \cite{gro_frm}, Thm. 0.1.A).
  To show  (\ref{MAINest}), let   $\gamma_1$ be a shortest non-nullhomotopic closed geodesic  realizing the systole of $X$,  let $x_0  \in \gamma_1$ and let $g_1$ be the class of $\gamma_1$ in $\pi_1(X, x_0)$. 
  Consider the natural action by deck transformations of $G=\pi_1(X, x_0)$  on the Riemannian universal covering  $\tilde X$, and the \textit{displacement function} of $G$ on $\tilde X$ 
 
 $$ \Delta_{G}(\tilde x)  := \inf_{g\in G^\ast} d(\tilde x, g. \tilde x)$$
 
\noindent whose infimum over $\tilde X$  coincides with   $\sys \, \pi_1(X)$, and is realized by $g_1$ at any preimage $\tilde x_0\in\tilde X$ of $x_0$.
\noindent Then, consider the finite generating set of $G$ given by     (cp.   \cite{Gromov})  
 $$\Sigma =\{g\in G\,|\,  d(\tilde x_0 ,g.\tilde x_0 )\le2D\}.$$ 
 We shall  consider separately the cases where $g_1$ is  elliptic or  hyperbolic. \\
  If $g_1$ is elliptic, we know by Lemma \ref{gensys} (a)  that   there exists  a hyperbolic element $h$ with $|h|_\Sigma \leq 2$. Setting $g_2=h^pg_1h^{-p}$, for the least integer $p \geq (k+1)/2$, the elements $\{g_1,g_2\}$ generate a nonabelian free subgroup, by Theorem \ref{freesmgr} (ii).  \\
  We now use the following Lemma, which is folklore (see for instance  \cite{Cer}):

\begin{lem*}
Let $\mathbb{F}_2$ be a  free nonabelian group, freely generated by $\Sigma=\{g_1,g_2\}$.\linebreak
  For any   word-weighted distance   $d_\ell$  on the Cayley graph $\mathcal C(\mathbb{F}_2,\Sigma)$, defined by the conditions $|g_1|_\ell=\ell_1$ and $|g_2|_\ell=\ell_2$, the entropy $\mathcal E=\Ent(\mathbb{F}_2,d_\ell)$ solves the equation: 
\begin{equation}\label{cer_ent}
(e^{\mathcal E\cdot \ell_1}-1)(e^{\mathcal E\cdot \ell_2}-1)=4
\end{equation}
\end{lem*}

\pagebreak
 Applying this lemma to $\mathbb{F}_2 \cong \langle g_1, g_2\rangle$, endowed with the word-weighted  distance $d_{\ell}$ defined by  $\ell_1 := | g_1|_{\tilde x_0} = \sys \, \pi_1(X) $ and $\ell_2:=  |g_2|_{\tilde x_0} \le (4k+10)D$, we derive from   equation (\ref{cer_ent}) : 
 
\begin{equation}\label{minorsys}
 \ell_1 \ge\frac{1}{\mathcal E}\,\cdot\,\log\left(1+\frac{4}{e^{\ell_2  \cdot D \mathcal E}-1}\right)\ge\frac{1}{ E}\,\cdot\,\log\left(1+\frac{4}{e^{(4k+10)  \cdot D  E}-1}\right)
\end{equation}
 
since $d_{\tilde x_0} \le d_\ell$ and so,  by  (E1) and (E2),
$$ \mathcal E= \Ent(\langle g_1, g_2\rangle, d_{\ell}) 
     \le \Ent(\langle g_1, g_2\rangle, d_{\tilde x_0})
     \le \Ent(G, d_{\tilde x_0})
     \le \Ent(X) = E$$
 
\noindent This concludes the proof in the case where $g_1$ is elliptic.\\
Assume  now that $g_1$ is a hyperbolic element. By Lemma \ref{gensys}, we can pick an element $s \in \Sigma$ which is not in $N_G(g_1)$. By  the discussion in Lemma \ref{gensys},  either $s$ is hyperbolic with $\ax(s)\neq \ax(g_1)$, or $s$ is elliptic and does not preserve $\ax(g_1)$. \linebreak
In the first case, we deduce by Theorem \ref{freesmgr} (iii)  that $\{ g_1, g_2\}$ generate a  free nonabelian semigroup of rank 2, for some choice of $g_2 \in \{ s, s^{-1} \}$.   In the second case, $g_2:=sg_1s^{-1}$ is a hyperbolic element with $\ax(g_2)\neq \ax(g_1)$ and, by the same theorem,   $\{ g_1, g_2\}$ generate a free nonabelian semigroup. \\
\noindent We can now use the following (see \cite{BCG}, Lemme 2.4):
\begin{lem} \label{stima}
Let $\mathbb{F}^+_2$ be a nonabelian  semigroup, freely generated by $\Sigma=\{g_1,g_2\}$.
For any  left invariant distance $d$ on $\mathbb{F}^+_2$ and  any choice of  of positive real numbers $(\ell_1,\ell_2)$ such that $|g_1|_d \le \ell_1$ and $|g_2|_d \le \ell_2$,   the entropy $\mathcal E=\Ent(\mathbb{F}^+_2,d)$ satisfies the inequality:
$$\mathcal E=\ent(\mathbb{F}^+_2,d)\ge\sup_{a\in(0+\infty)}\left(\frac{1}{\ell_1+a\ell_2}\right)\cdot ((1+a)\cdot\log(1+a)-a\log(a))$$
 
\end{lem} 

\noindent 
We apply this lemma to  $\mathbb{F}^+_2 \cong \langle g_1, g_2\rangle$, 
for    $\ell_1 := |g_1|_{\tilde x_0}$ and $\ell_2:= |g_2|_{\tilde x_0}\le 6D$,
and we derive, by choosing  $a= E\cdot \ell_1 $
\begin{equation}\label{minorsys2}
 \ell_1 \ge \frac{1}{E}\,\cdot\, e^{-6D E}
\end{equation}
since  $\log(1+a)\ge\frac{a}{1+a}$ and   $\mathcal E\le E$. 
\vspace{5mm}

\noindent If $k \geq 1$, this lower bound for the systole is greater than the one in (\ref{minorsys}) 
(actually,  the inequality  $e^{-6x}<\log\left(1+\frac{4}{e^{(4k+10)x}-1}\right)$ implies that $x \leq \frac{21}{125}$, and in this case $2x<e^{-6x}$; but if $x=ED \leq \frac{21}{125}$ then $\ell_1 \cdot E \leq 2DE < e^{-6DE}$, contradicting  (\ref{minorsys2})). 
On the other hand, if $k=0$ the stabilizers of the edges of $\mathcal T$ are trivial and thus $G$ splits as a free product of a finite number of finitely generated, torsionless groups. By \cite{Cer}, Theorem 1.3, the following  estimate 
 for the systole of finitely generated,  torsionless free products holds:
 
$$\sys\pi_1(X)\ge\frac{1}{E}\cdot\log\left(1+\frac{4}{e^{2D\,E}-1}\right)$$
which is sharper than (\ref{minorsys}) and concludes the proof of Theorem \ref{MAINTH}.
\end{proof}

\pagebreak

\section{Applications to  $3$-manifolds}\label{sectiondim3}

 \noindent The section is devoted to the proof of Theorems \ref{app_MTHM_1}\&\ref{rigidity}, and of their corollaries.\\
In \S4.1,  we recall some basic results of $3$-dimensional topology (the Prime Decomposition and the JSJ-decomposition) and  prove that given a compact $3$-manifold $X$ without spherical boundary components, either $int(X)$ admits a geometric metric, or $\pi_1(X)$ has a splitting  as a free or amalgamated product  which is $4$-acylindrical.  \\
In \S4.2,    as a consequence of this dicothomy and of Theorem \ref{MAINTH}, we shall obtain the systolic and volume estimates (Theorem \ref{app_MTHM_1} and Corollary \ref{app_MTHM_1v}), and  we shall prove the rigidity results (Theorem \ref{rigidity} and Corollaries \ref{topfiniteness}, \ref{diffrigidity} \& \ref{fnt_glo}). 

\subsection{Acylindrical splittings of   non-geometric, $3$-manifolds groups} ${}$\\
For a comprehensive exposition of the topics that we recall here,  we refer to the classical books of Hempel and Thurston (\cite{hempel}, \cite{Thurst}), to the survey papers of Scott and Bonahon (\cite{Scott}, \cite{Bon}) and to the recent monography of Aschenbrenner, Friedl and Wilton (\cite{AFW}). 
\vspace{1mm}

We recall that a compact $3$-manifold $X$ is said to be \textit{prime} if it cannot be decomposed  non trivially as the connected sum of two manifolds, \textit{i.e.} when $X=X_1\# X_2$ then either $X_1$ or $X_2$ is diffeomorphic to $\mathbb{S}^3$.   A compact $3$-manifold $X$ is called \textit{irreducible} if every embedded $2$-sphere in $X$ bounds a $3$-ball in $X$ (and {\em reducible}  otherwise). Every orientable,  irreducible  $3$-manifold is  prime; conversely, if $X$ is an orientable, prime  $3$-manifold  with no spherical boundary components, then either $X$ is irreducible, or $X=\mathbb{S}^1\times \mathbb{S}^2$ (see \cite{hempel}, Lemma 3.13).  
Notice that an irreducible, orientable, compact $3$-manifold 
does not have boundary  components homeomorphic to the $2$-sphere, unless the manifold is the $3$-ball. \\
As we  deal also with compact $3$-manifolds $X$ with possibly non-empty  boundary we need a few more definitions: 
an embedded surface $S\subset X$ is said to be {\em incompressible} if for any embedded disk $D\subset X$  with  $\partial D\subset S$ there exists a disk $D'\subset S$ such that $\partial D'=\partial D$; when $X$ is irreducible, this implies that the disk $D$ is isotopic to $D'$. 
In particular,  $X$ has \textit{incompressible boundary} if any connected component of $\partial X$ is an incompressible surface.
Finally, a {\em $\partial$-parallel} properly embedded surface of $X$ is an embedded surface $S$ whose (possibly empty) boundary is contained in $\partial X$ and such that $S$ is isotopic $\mathrm{rel}\, \partial X$ to a subsurface in $\partial X$.


 A cornerstone of $3$-dimensional topology is the 

\begin{PdT}
	Let $X$ be any compact, oriented $3$-manifold.
	There exist oriented, prime, compact $3$-manifolds $X_0$, $X_1$,..., $X_m$ such that $X_0$ is diffeomorphic  to a sphere minus a finite collection of disjoint $3$-balls, $X_i$ has no spherical boundary components for $i \geq 1$, and $X = X_0\#X_1\#\cdots \#X_m$. \\
	Moreover,   if  $X_i'$, for $i=0,...,m'$, are manifolds with the same properties as the $X_i$'s, and 
	 $X = X_0\#X_1\#\cdots\# X_{m} =  X_0'\#X_1'\#\cdots\# X_{m'}'$ 
	  then $m=m'$ and (possibly after reordering the indices) there exist orientation-preserving diffeomorphisms $X_i \! \stackrel{\sim}{\f}  \! X_{i}'$.
	  The manifolds $X_i$ are   the {\em prime pieces} of $X$.  
\end{PdT}

The Prime decomposition Theorem has a partial converse, the Kneser's conjecture. In  classical references, the conjecture is stated for closed $3$-manifolds
or compact $3$-manifolds with incompressible boundary; actually, the conjecture  
is false in presence of compressible boundary, exceptly in case where the compressible boundary components are tori:

\begin{kneser}\label{Kneser}
	Let $X$ be any compact $3$-manifold whose  compressible boundary components (if any) are homeomorphic to tori.  
	If  $\pi_1(X) = G_1\!*\! \cdots \!*\!G_n$, \linebreak then there exist  compact $3$-manifolds  $X_1$,..., $X_n$, such that $\pi_1(X_i)= G_i$ and \linebreak $X= X_1\# \cdots  \# X_n$.
\end{kneser}

For compact irreducible $3$-manifolds there exists a second important decomposition theorem, due to   the independent work of Jaco-Shalen (\cite{jacosh}) and Johannson (\cite{jo1}, \cite{jo2}): this decomposition is obtained by cutting along embedded incompressible tori, which split the manifold into elementary pieces which are of two different (but not mutually exclusive) kinds: atoroidal  pieces and Seifert fibered pieces. We recall that a compact, irreducible $3$-manifold $X$ is said to be \textit{atoroidal} if any incompressible torus is  $\partial$-parallel. A compact, irreducible $3$-manifold is said to be a {\em Seifert fibered manifold} if it admits a decomposition into disjoint simple closed curves (the fibers of the Seifert fibration) such that each fiber has a tubular neighborhood which is isomorphic, as a circle bundle, to a \textit{standard fibered torus}\footnote{A pair of integers $(a,b)\in\N^*\times\Z$ being given, the associated standard fibered torus $\mathbb{T}_{a,b}$  is the circle bundle over the  disk $D^2$ obtained from $D^2 \times [0,1]$ by identifying the boundaries $D^2 \times \{0\}$ with $D^2 \times \{1\}$ via the automorphism  $\varphi: D^2 \rightarrow D^2$ given by the rotation by an angle of $\frac{2\pi\,b}{a}$; this manifold comes naturally equipped with a fibering by circles, given by gluing the ``parallels'' $\{p\} \times [0,1]$ of $\mathbb{T}_{a,b}$ via $\varphi$.    
 }.

\begin{JSJdT}
	Let $X$ be a compact, orientable, irreducible $3$-manifold. There exists a (possibly empty) collection of disjointly embedded incompressible tori $T_1,\dots, T_m$ such that each component of $X\setminus\bigcup_{1}^mT_i$ is atoroidal or Seifert fibered. A collection of tori with this property and having  minimal cardinality is unique up to isotopy. 
\end{JSJdT}

\noindent We shall refer  to  the minimal collection of tori   $\{T_1,\dots, T_m\}$  as to the \textit{\JSJ-tori of $X$}, and   to the connected components of $X$ cut along $\bigcup_{i=1}^mT_i$ as to the \textit{\JSJ-components of $X$};  the $JSJ$-decomposition is said {\em trivial} when the collection of JSJ-tori is empty.\\
As we remarked, Seifert fibered $3$-manifolds can be atoroidal: the list of atoroidal Seifert fibered $3$-manifolds can be found in Jaco-Shalen (\cite{jacosh}, IV.2.5, IV.2.6).
Following Thurston \cite{thur82} we say that an irreducible $3$-manifold $X$ is \textit{homotopically atoroidal} if every $\pi_1$-injective map from the torus to $X$ is homotopic to a map into the boundary; using Jaco-Shalen terminology this means that a manifold $X$ does not admit a non-degenerate map $f:T^2\f X$. Being homotopically atoroidal is a stronger property than just being atoroidal (as one allows continuous maps which are not embeddings);  however, the two notions coincide outside of Seifert fibered manifolds. The list of compact, homotopically atoroidal, orientable Seifert fibered manifolds is the following: Seifert fibered manifolds with finite fundamental group, $S^2\times S^1$, $D^2\times S^1$, $T^2\times I$ and the twisted, orientable interval bundle over the Klein bottle $K\widetilde \times I$; we observe that  only the last three have non-empty boundary. 
\vspace{1mm}

Following again \cite{thur82}, we define:

\begin{defn}\label{nnGEO}
	Let $X$ be a compact  $3$-manifold with (possibly empty) boundary.
	  We say that $X$ is {\em non-geometric} if its interior cannot be endowed with a complete metric which is locally isometric to one of the eight model geometries.
\end{defn}

\pagebreak

The geometrization of closed, orientable Seifert fibered $3$-manifolds $S$ is explained  in \cite{Scott}; on the other hand, the geometrization  of Seifert fibered manifolds with  boundary can be found in  \cite{Bon}  (where the  geometrization is meant with totally geodesic boundary;  the geometrization in Thurston's sense, i.e. with complete, geometric metrics,  is obtained   from a Fuchsian representation of the orbifold fundamental group of the base space  with parabolic  boundary generators,  and then extending it to a representation of $\pi_1(S)$ in $\mathrm{Isom}_+(\mathbb H^2\times \R)$, as explained in \cite{ohshika}). For the remaining three Seifert fibered manifold, the interior of $K\widetilde\times I$, $D^2\times I$ and $T^2\times I$ can be endowed with complete euclidean metrics.
For the remaining three Seifert fibered manifold, the interior of $K\widetilde\times I$, $D^2\times I$ and $T^2\times I$ can be endowed with complete euclidean metrics.\\ 
For what concerns the atoroidal pieces, 
Thurston's Hyperbolization Theorem
\footnote{Thurston announced for the first time in 1977 his Hyperbolization Theorem, and in 1982  the Geometrization Conjecture  \cite{thur82}; in the series of papers \cite{thur86a}, \cite{thur86b}, \cite{thur86c} (the latter two of which unpublished) Thurston filled some of the major gaps. Complete proofs can be found in \cite{otal96}, \cite{otal98}, \cite{kap}.}
asserts that a closed, Haken $3$-manifold admits a complete hyperbolic metric if and only if it is homotopically atoroidal, and that the interior of a compact, irreducible $3$-manifold with non-empty boundary can be endowed with a complete hyperbolic metric if and only if it is homotopically atoroidal and not homeomorphic to $K\widetilde{\times}I$. \\
On the other hand, the fact that closed, irreducible, homotopically atoroidal {\em  non-Haken} $3$-manifold admit a geometric metric is the content of Thurston's Geometrization Conjecture, proved by Perelman (\cite{per1}, \cite{per2}, \cite{per3}). In particular, the Elliptization Theorem shows that closed $3$-manifolds with finite fundamental group are finite quotients of $S^3$ (and thus Seifert fibered), and the Hyperbolization Theorem for the non-Haken case shows that irreducible, homotopically atoroidal, non-Haken $3$-manifolds carry complete hyperbolic metrics (for more references and further readings see \cite{AFW},  Ch.1, \S7). 
\vspace{1mm}

In view of this discussion, and for future reference, we record the following, now  well-established



\begin{fatto}\label{geometrization} $\!\!\!\!$
	A compact,  irreducible $3$-manifold 
with trivial JSJ-decomposition  is geometric.
\end{fatto}




 Given a compact $3$-manifold $X$, we shall call  the splitting of the fundamental group of $X$ as a graph of groups induced by the prime decomposition of $X$, or by   the JSJ-decomposition (when $X$ is irreducible)  the  {\em canonical splitting} of $\pi_1(X)$.
We shall say that  $X$ has a \textit{non-elementary, canonical, $k$-acylindrical splitting} if the action of $\pi_1(X)$ on the Bass-Serre tree associated to the canonical splitting  is non-elementary and $k$-acylindrical.

\begin{dico}[Geometric vs acylindrical splitting]${}$\\
	Let $X$ be a compact, orientable $3$-manifold with no spherical boundary components.\\ Then, either $X$ is  geometric  or $\pi_1(X)$ has a non-elementary, canonical $4$-acylindrical splitting. The two possibilities are mutually exclusive.
\end{dico}

\begin{rmk}
The dichotomy clearly does not hold in presence of spherical   boundary (as excising  an arbitrary number of disjoint balls from a geometric manifold does not change the fundamental group). Moreover,  we stress the fact  that the above dicothomy does not assert that fundamental groups of geometric, compact $3$-manifold do not admit  acylindrical splittings, different from the canonical one, as we shall see in the Example \ref{Schottky}.
\end{rmk}

\begin{proof}[Proof of the dicothomy]
	Assume first that $X$ is a compact, orientable $3$-manifold, whose prime decomposition is non-trivial. 
 Then, $X$ has at least two non-simply connected prime pieces (because, since $X$ has no spherical boundary components, the first piece $X_0$ given by the prime decomposition is empty).
Then, either $X$ is homeomorphic to $\R P^3\#\R P^3$ or the action of $\pi_1(X)$ on the Bass-Serre tree associated to the prime splitting is non-elementary 
(since the action of any non-trivial free product different from  $\Z_2 \ast \Z_2$ on its Bass-Serre tree does not have any globally invariant line).
In the first case observe that $\R P^3 \#\R P^3$ is the unique orientable non-prime, Seifert fibered space (see \cite{AFW} pg. 10) and, in particular, admits a geometry  modelled on  $\mathbb{S}^2 \times \mathbb{R}$ (see \cite{Scott}). Otherwise, since  the edge stabilizers in the prime splitting are trivial  and at least one vertex group is different from $\mathbb{Z}_2$, the prime splitting is  $0$-acylindrical.	\\
	Let us assume now that $X$ is a prime, compact $3$-manifold; we may actually  assume that $X$ is irreducible,  as $\mathbb{S}^2\times \mathbb{S}^1$ is  geometric.  If the JSJ-decomposition of $X$ is trivial, then $X$ is geometric, in view of Fact \ref{geometrization}, and  the canonical splitting of $\pi_1(X)$ is elementary. 
On the other hand,  in \cite{WiZa} Wilton and Zalesskii prove that if $X$ is a closed, orientable, irreducible $3$-manifold, then  either $X$ admits a finite sheeted covering space that is a torus bundle over the circle, or the  JSJ-splitting is $4$-acylindrical. 
The same result holds for compact, irreducible  manifolds (see for details  \cite{cer_pi1}, where the precise constants  of acylindricity  of the splitting of $\pi_1(X)$ as an amalgamated or a HNN-extension over the peripheral groups is computed, according to the different types of the adjacent JSJ-components).\\
Now, compact, orientable, irreducible $3$-manifolds with non-trivial JSJ-decompo\-si\-tion, which are finitely covered by a torus bundle, are either equal to a \textit{twisted double} $D(K\widetilde \times I, A)$ or to a \textit{mapping torus} $M(T^2,A)$, for a gluing map $A\in\mathrm{SL_2(\Z)}$ such that, respectively, $JAJA^{-1}$ and  $A$ are Anosov 
(where $ J(x,y)=(-x,y)$,
 see Theorems 1.10.1, 1.11.1 in \cite{AFW}). In both cases the resulting manifolds admit a $Sol$-metric (Theorem 1.8.2 \cite{AFW}), hence they are geometric. \\
It remains to show that the $4$-acylindrical splitting is non-elementary. Actually as $X$ has a non-trivial JSJ-decomposition,  it is clear that the  action of $\pi_1(X)$ is not elliptic;  moreover, if it was linear then $\pi_1(X)$ would be virtually cyclic, by Lemma \ref{gensys}, which  contradicts the fact that $\pi_1(X)$ contains a rank $2$ free abelian subgroup.
\end{proof}

\subsection{Systolic and volume estimates, local rigidity and finiteness}  
${}$

\vspace{-3mm}
\begin{proof}[Proof of Theorem \ref{app_MTHM_1}] In view of the above Dicothomy,    $\pi_1(X)$ admits a non-ele\-men\-tary, canonical $4$-acylindrical splitting. By assumption,  $\pi_1(X)$ is torsionless, so we can apply 
Theorem \ref{MAINTH} to deduce 
	\small
	$$\sys\pi_1(X)\ge\frac{1}{E}\log\left(1+\frac{4}{e^{26\,E\,D}-1}\right)=s_0(E,D)$$
	\normalsize
\end{proof}

\begin{proof}[Proof of Corollary \ref{app_MTHM_1v}]
Let $X\!=\! X_0 \#  \cdots \#X_m$ be the prime decomposition of $X$. \linebreak
Since $X$ is closed and different from $\#_k (\mathbb{S}^2 \times \mathbb{S}^1)$, the  piece $X_0$ is empty and there exists at least a prime piece, say $X_1$, which is closed and irreducible. Moreover, since $X$ has torsionless fundamental group, $X_1$ is aspherical, and the existence of a degree one projection map $X \rightarrow X_1$ shows that $X$ is $1$-essential. 
Since we know that the systole of $X$ is bounded below by $s_0 (E,D)$,  we can apply Theorem 1.0.A. in \cite{gro_frm} to obtain the estimate $\Vol(X)\ge C\cdot s_0(E,D)^3$.
\end{proof}

\begin{proof}[Proof of Theorem \ref{rigidity}]
 Consider $X, X'\in\mathscr M_{ngt}^{\partial}(E,D)$. By Theorem \ref{app_MTHM_1} we know that the systoles of $X$ and $X'$ are bounded below by $s_0(E,D)$; then, also their  semi-locally simply connectivity radius $r(X_i)$ 
 \footnote{The semi-locally simply connectivity radius of a $X$ is the supremum of $r$ such that every closed curve in a ball of radius $r$ is homotopic to zero in $X$.}
 is bounded below by $\frac12 s_0(E,D)$. 
 Now, two compact Riemannian manifolds with $d_{GH}(X_1,X_2) < \frac{1}{20}\min\{r(X_1), r(X_2)\}$  have isomorphic fundamental group, as proved by  Sormani and Wei \cite{sowei} (as a consequence of  \cite{tusch}, Theorem (b)).
This  proves (i). 
%
To show (ii), assume  that, moreover,  $X$ and $X'$ are irreducible:  since their fundamental group is torsionless, they are aspherical, and then homotopy equivalent by Whitehead's Theorem.

\end{proof}

\begin{proof}[Proof of Corollary \ref{topfiniteness}]
 By Theorem  \ref{rigidity} (i) we know that given $X \in \mathscr M_{ngt}^{\partial}(E,D)$,  there exists a $\delta_0=\delta_0(E,D)$ such that every other manifold $X'$ in $\mathscr  M_{ngt}^{\partial}(E,D)$ which is $\delta_0$-close to $X$ has the same fundamental group as $X$.
Now,  recall that, by  results of Swarup  \cite{swa}, there is a finite number of irreducible, compact $3$-manifolds with a given fundamental group.
By the Prime Decomposition Theorem (as stated in Section \S4.1), and by uniqueness of the decomposition of a group as a free product, this is also true for (possibly reducible) compact $3$-manifolds, without spherical boundary components  (recall that $\mathbb{S}^2 \times \mathbb{S}^1 $ is the only prime, \linebreak not irreducible, orientable manifold without spherical boundary components).  \linebreak
We then conclude that the ball at $X$ of radius $\delta_0$ in $\mathscr M_{ngt}^{\partial}(E,D)$ contains only a finite number of homeomorphism (and then diffeomorphisms) types.
\end{proof}

 Corollary \ref{diffrigidity} is a particular case of the following:

\begin{prop}\label{diffrigiditytech}
	Let $X, X'\in\mathscr M_{ngt}^{\partial}(E,D)$ with $X$ be irreducible. Assume that $d_{GH}(X,X')<\delta_0$, for  $\delta_0=\delta_0(E,D) $ as in Theorem \ref{rigidity}:\\
(i) if $\partial X$ is incompressible, then $X'$ is homotopy equivalent to $X$;\\
(ii) if $\partial X=\varnothing$, then $X$ is diffeomorphic to $X'$.

\end{prop}

\begin{proof}
	Let us prove (i). 
	By Theorem \ref{rigidity} (i) we deduce that $\pi_1(X)\cong\pi_1(X')$, and this group is indecomposable, by Kneser's Conjecture.
	As $X'$ has no spherical boundary components, it follows from the Prime decomposition Theorem that $X_0'$ is empty and $X'=X_1'$; better, since it is not geometric, it is different from $\mathbb{S}^2 \times \mathbb{S}^1$ and so it is irreducible too. We can then apply  Theorem \ref{rigidity} (ii) to deduce that $X'$ is homotopically equivalent to $X$. \\
	For (ii), we  deduce as in (i) that $X'$ is  homotopy equivalent to $X$, an then closed. This implies that $X'$ is homeomorphic (and actually diffeomorphic) to $X$, by the discussion after Corollary \ref{diffrigiditytech} in Section \S1.
\end{proof}


\begin{proof}[Proof of Corollary \ref{fnt_glo}] 
By Bishop's comparison theorem it follows that the space $\mathscr M_{ngt}(Ric_K,D)$ is included in $  \mathscr M_{ngt}(2K,D)$. Moreover, Gromov's precompactness theorem asserts that the family $\mathscr M_{ngt}(Ric_K,D)$ is precompact; therefore, for any arbitrary $\delta>0$,  this space can be covered by a finite number of balls of radius $\delta$. Taking $\delta = \delta_0(2K,D)$, where $\delta_0$ is the function in Theorem \ref{rigidity},  and using Corollary \ref{topfiniteness}  we infer the finiteness of the diffeomorphism types in   $\mathscr M_{ngt}(Ric_K,D)$.
\end{proof}

\pagebreak

\begin{rmk} {\em Is the peripheral structure preserved by Gromov-Hausdorff approximations?}
	We recall that the \textit{peripheral structure} of a $3$-manifold $X$ with incompressible boundary is the data of the fundamental group $\pi_1(X)$ together with the collection of the conjugacy classes of  subgroups determined by the boundary components.  \linebreak
	 Let $X_1$ and $X_2$ be two compact, orientable, irreducible $3$-manifolds with non-spherical,  incompressible boundary.  Waldhausen (\cite{Wal}) proved that any isomorphism $\varphi:\pi_1(X_1)\f\pi_1(X_2)$  sending the peripheral structure of $X_1$ into the peripheral structure of $X_2$   is induced by a homeomorphism. It is not known to the authors if 
 the isomorphism between the fundamental groups induced from  a Gromov-Hausdorff $\varepsilon$-approximation $f: X_1 \leftrightarrows X_2$, with $\epsilon$ sufficiently small,   preserves the peripheral structure. If this was the case, then Corollary \ref{diffrigidity} would hold for all non-geometric, irreducible manifolds with (possibly empty) incompressible boundary.
\end{rmk}

\section{Examples}
	We give here a collection of examples (which do not satisfy the assumptions of Theorems \ref{app_MTHM_1}, \ref{app_MTHM_1v}), where the systole or the volume can be collapsed while keeping entropy and diameter bounded.
	
	\begin{exmp}\label{excoll} {\em Collapsing the systole of geometric $3$-manifolds.}\\ 
		For each  model geometry  different from  $\mathbb H^3$, we can exhibit a closed Riemannian manifold $X$ and a sequence of metrics $h_{\mathbb G}^{\varepsilon}$, for $\varepsilon\in(0,1]$,  such that $\ent(X, h_{\mathbb G}^{\varepsilon})\le E$, $\diam(X, h_{\mathbb G}^{\varepsilon})\le D$ and $\sys\pi_1(X, h_{\mathbb G}^{\varepsilon})\f 0$.\\
		This is trivial for ${\mathbb G} = {\mathbb S}^3, {\mathbb S}^2 \! \times \! {\mathbb R}, {\mathbb E}^3$ and  $Nil$, which have sub-exponential growth: just take the standard sphere, ${\mathbb S}^2 \times {\mathbb S}^1$, any flat torus $T$, and the quotient $H^3_{\mathbb Z} \backslash Nil$ of the Heisenberg group by the  standard integral lattice, and scale the model metric by $\epsilon$. The systole and diameter collapse, while the entropy is always zero.\\ 		
		\noindent For  ${\mathbb G} =   \mathbb H^2\times\mathbb R$, $ \mathbb H^2\widetilde\times\mathbb R$, we can just take the Riemannian product  $X=S_g \times {\mathbb  S}^1$ of a closed hyperbolic surface $S_g$    of genus $g\ge 2$ with the circle, and the  unitary tangent bundle $X=U S_g$ of $S_g$ with its Sasaki metric; then, we contract by $\epsilon$  the model metrics $h_{\mathbb G} $ along the fibers of the ${\mathbb  S}^1$-fibration $X \rightarrow S_g$. In both cases, the  sectional curvature of  the new metrics $h^\epsilon_{\mathbb G} $ stays bounded, 
		as  $X$ admits a free, isometric action of ${\mathbb  S}^1$ along the fibers   (a pure, polarized  $F$-structure, cp.  \cite{chegro}); thus, the entropy is bounded uniformly, while the systole collapses ( and $X$ tends to $S_g$ in the Gromov-Hausdorff distance). \\ Notice that, in the second case, the collapse is through non-model metrics.

		
		\noindent  In the last case consider the group ${\mathbb G} =Sol$, defined, for any hyperbolic endomorphism $A \in SL(2,{\mathbb Z})$ with eigenvalues $\lambda^{\pm1}$,  as the semidirect product  ${\mathbb R} ^2 \rtimes_A {\mathbb R} $, with   ${\mathbb R} $  acting on ${\mathbb R} ^2$ as   $A^t$, and endowed with the canonical left-invariant metric (in the diagonalizing coordinates $(x,y)$):
		$$\mathrm{h}_{Sol}=\lambda^{2t}\,dx^2\oplus \lambda^{-2t}\,dy^2\oplus dt^2$$ 
		Consider the quotient $X^\epsilon$ of $Sol$ by the discrete subgroup of isometries $\Gamma^{\varepsilon}$  generated by the lattice $\epsilon {\mathbb Z} ^2$ (acting by translations on the $xy$-planes)  and by the isometry $s (u, t)\mapsto (A u, t+1)$.
		The manifolds $X^\epsilon$ are diffeomorphic, with  $\sys(X^\epsilon) \rightarrow 0$ and bounded diameter; on the other hand, they all have isometric universal covering, thus  $\ent(X^\epsilon, h_{Sol})$ is equal to the exponential growth rate of $\ent(Sol, \mathrm{h}_{Sol})$ for all $\varepsilon\in(0,1]$.  
	\end{exmp}

%
\pagebreak

\begin{exmp}{\em Collapsing the volume of the connected sum $\#_k({\mathbb S}^2\times {\mathbb S}^1)$.}${}$
	\label{collapsingS2S1}

		\noindent We shall construct, for $\varepsilon\in(0,1]$,  a family of metrics $g_{\varepsilon}$  on the connected sum $kX= X_1 \#\cdots \# X_k$ of $k$ copies  $X_i={\mathbb S}^2\times {\mathbb S}^1$, such that $\sys\pi_1(kX,g_{\varepsilon})\sim 2\pi$, $\diam(kX, g_{\varepsilon})\le D$, $\ent(kX, g_{\varepsilon})\le E$ for all $\epsilon$, while the volume  goes to $0$ as $\varepsilon\f 0$. \\
		Consider the canonical product metric $h=h_{\mathbb S^2}\oplus h_{\mathbb S^1}$ on ${\mathbb S}^2\times {\mathbb  S}^1$.
		We construct  $g_{\epsilon}$ by scaling $h_{\mathbb S^2}$ by $\epsilon$ and gluing the $k$ copies of ${\mathbb S}^2\times {\mathbb  S}^1$ through a thin, flat cylinder. \linebreak Namely,  two base points   $x_i^\pm$ on $X_i$ being chosen (with $x_1^+=x_1^-$ and $x_k^+=x_k^-$),  let  $ h_\epsilon = \epsilon^2 h_{\mathbb S^2}\oplus h_{\mathbb S^1}$    and let $r_\epsilon = \inj ( {\mathbb S}^2\times {\mathbb  S}^1,  h_\epsilon )$. We  write the metric in each copy in polar coordinates around $x_i^\pm$ as
		$$  h_\epsilon  = \varphi_\epsilon^2 (r,u) h_{\mathbb S^2} + dr^2$$
		and modify $h_\epsilon$ around the points $x_i^\pm$ into a new metric $\tilde{h}^i_\epsilon$ on  $X_i \setminus \{ x_i^\pm \}$,   which interpolates,  on the annulus $B_{h_\epsilon } (x_i^{\pm},  \epsilon r_\epsilon) \setminus B_{h_\epsilon } (x_i^{\pm}, \epsilon^2 r_\epsilon)$,  between $  h_\epsilon$ and the  product metric $(\epsilon^2r_\epsilon)^2  h_{\mathbb S^2} + dr^2$ of the cylinder $(\epsilon^2r_\epsilon) {\mathbb S}^2 \times {\mathbb S}^1$; 
		finally, we glue the copies $(X_i \! \setminus \!  \{x_i^{\pm}\}, \tilde h^i_\epsilon)$ and $(X_{i+1} \! \setminus \!  \{x_{i+1}^{\pm}\}, \tilde h_\epsilon^{i+1})$ to obtain $(kX,g_\epsilon)$, by identifying the flat  $\epsilon^2 r_\epsilon$-annulus around $x_i^-$ to the corresponding annulus around $x_{i+1}^+$ via an  isometry interchanging the boundaries.\\
		It is then easy to check that the manifolds $(kX, g_{\varepsilon})$  converge in the Gromov-Hausdorff distance to the length space given by the wedge $X_0=\vee_{x_i,...,x_k} {\mathbb S}^1$ of $k$ copies of  the standard circle $  {\mathbb S}^1$ with respect to appropriate points $x_1,...,x_k$. Notice that by construction we have  $\diam(kX, g_{\varepsilon})\le  k \pi+1$,  that the systole of $(kX, g_{\varepsilon})$ is  bounded from below by $2\pi-1$ for all sufficiently small  $\epsilon$,  and that 
		clearly  $\Vol(kX, g_{\varepsilon})\f 0$. Moreover,  the entropy of all these manifolds is uniformly bounded from above by $Ent(X_0) +1$, for  $\epsilon \rightarrow 0$;  this follows for instance from \cite{rev}, Proposition 38.
	
\end{exmp}

	Finally, we give examples  of $3$-manifolds with different topology, which  are arbitrarily close in the Gromov-Hausdorff distance, while satisfying entropy and diameter uniform bounds. 
	
 \begin{exmp}{\em Manifolds with spherical boundary components.\label{nonnprig} ${}$} ${}$\\
Take any  closed, irreducible Riemannian  $3$-manifold $X$ with  $\sys \pi_1(X) \ge1 $, and remove a disjoint collection of $n$ balls $B(x_i,\epsilon)$, for arbitrarily small $\epsilon$.  
		The resulting, reducible manifold $X_{n,\epsilon}$ with spherical boundary has  the same fundamental group as $X$, while being not   homotopically equivalent to $X$.  $X_{n,\epsilon}$ clearly is $(2n \pi \epsilon)$-close to $X$,  as the metric on a sufficently small ball around $x_i$ can be approximated by the Euclidean one; hence $\diam(X_{n,\epsilon}) \leq \diam (X) + 2n \pi \epsilon$ too.
		It is easy to verify that, for small values of $\epsilon$, the orbits of $G=\pi_1(X) \cong \pi_1(X_{n,\epsilon})$, on the respective Riemannian universal converings,  are $\left(1+ \frac{3n\pi}{ \sys (X)} \right)$-biLipschitz  to each other; this implies the entropy bound $Ent(X_{n,\epsilon}) \leq (1+ \frac{3n\pi}{ \sys (X)}) Ent(X)$.  
		\vspace{1mm}
		\end{exmp}	
		
\begin{exmp}{\em Connected sums of hyperbolic manifolds.\label{nonnprig} ${}$} ${}$\\			
		\noindent   Let $(X,h)$ be a closed hyperbolic $3$-manifold with  no orientation reversing isometries  (see \cite{mullner}), 
		 and 
		denote by $\overline X$  the same hyperbolic manifold endowed with the opposite orientation. 
		We know by standard differential topology that $X\# X$ and $X\# \overline X$ are not diffeomorphic; hence, by the discussion in Section \S1, they are not even homotopically equivalent.
		Now, remove from  $X$ and $\overline X$ small geodesic balls $B_h(x_0, \epsilon)$  of radius $\epsilon \ll \inj(X)$. 
		As in the Example \ref{collapsingS2S1},  we  modify the metric $h$ around  $x_0$ into a new metric $h_\epsilon$ which interpolates, on the annulus \linebreak $B_{h} (x_0, \epsilon) \setminus B_{h} (x_0, \epsilon^2)$, 
		between $h$ and the  product metric $\epsilon^4 h_{\mathbb S^2} + dr^2$;
		then, we glue together  the two copies of $(X  \! \setminus \!  \{x_0^{\pm}\}, h_\epsilon)$ 
		by identifying the two  cylinders $S^2 \times (\epsilon^2 , 0)$ via  an  orientation-reserving   (resp.  orientation-preserving)  isometry interchanging the boundaries,  to obtain a  Riemannian connected sum $Y_\epsilon=(X \#X, g_\epsilon)$   (resp.  $\bar Y_\epsilon =(X \# \bar X, \bar g_\epsilon)$.
		Then,  it is easy to show that both manifolds tend in the Gromov-Hausdorff topology to the length space given by the metric wedge $X \vee_{x_0} \! X$; hence they are arbitrarily close to each other for $\epsilon \rightarrow 0$, with diameters bounded by  $2\diam(X)+1$.
		Moreover, the systoles is uniformly bounded from below by $\sys\pi_1(X)/2$, so by \cite{rev}, Proposition 38, we deduce  that their entropies converge to  $Ent( X \vee_{x_0} \! X)$ and are uniformly bounded.
	\end{exmp}


\begin{exmp}{\em Hyperbolic manifolds with acylindrical splittings}\label{Schottky} \\
A \textit{handlebody} $H_g$ of genus $g> 0$ is,   topologically, the  $\varepsilon$-neighbourhood in $\R^3$  of a wedge  sum of $g$  circles; handlebodies  are classified by their genus.   The boundary of $H_g$ is an orientable, closed surface of genus $g$, and $\pi_1(H_g)\cong\mathbb F_g$; in particular, the fundamental group of $H_g$, for $g\ge 2$, is the non-trivial free product of $g$ infinite cyclic groups, hence it  admits a $0$-acylindrical splitting. It is not difficult to show that the interior of the handlebodies admits complete hyperbolic metrics: for $g \geq 2$,   it is sufficient to identify $H_g$ with the quotient of $\mathbb H^3$ by a Schottky  group of hyperbolic isometries,  generated by  $g$ hyperbolic translations,   with  disjoint axes and disjoint attractive and repulsive domains.
\end{exmp}

\end{document}